\documentclass[reqno]{amsart}
\usepackage{graphicx}
\usepackage{geometry}
\usepackage[leqno]{amsmath}
\usepackage{amscd}
\usepackage{amsthm}
\usepackage{amsfonts}
\usepackage{amssymb}
\usepackage[colorlinks=true]{hyperref}
\usepackage{xcolor}
\usepackage{mathrsfs}
\newtheorem{theorem}{Theorem}[section]

\newtheorem{corollary}[theorem]{Corollary}
\newtheorem{proposition}[theorem]{Proposition}

\theoremstyle{definition}
\newtheorem{definition}[theorem]{Definition}
\newtheorem{example}[theorem]{Example}

\makeatletter
\newcommand{\leqnomode}{\tagsleft@true}
\newcommand{\reqnomode}{\tagsleft@false}
\makeatother

\DeclareMathOperator{\End}{End}
\DeclareMathOperator{\Aut}{Aut}
\DeclareMathOperator{\Gal}{Gal}
\DeclareMathOperator{\Perm}{Perm}

\DeclareMathOperator{\FPF}{FPF}
\newcommand{\gen}[1]{\langle #1 \rangle} 

\newcommand{\B}{\mathfrak{B}}

\numberwithin{equation}{section}

\begin{document}
\newtheorem*{thm}{Theorem}

\title{Abelian fixed point free endomorphisms and the Yang-Baxter equation}
\author{Alan Koch}
\address{Department of Mathematics, Agnes Scott College, 141 E. College Ave., Decatur, GA 30030, USA \\ akoch@agnesscott.edu}

\author{Laura Stordy}
\address{School of Mathematics, Georgia Institute of Technology, 686 Cherry St NW, Atlanta, GA 30332, USA \\ lstordy3@gatech.edu }

\author{Paul J.~Truman}
\address{School of Computing and Mathematics, Keele University, Staffordshire, ST5 5BG, UK \\ p.j.truman@keele.ac.uk }
\date{\today       }

\begin{abstract} We obtain a simple family of solutions to the set-theoretic Yang-Baxter equation, one which depends only on considering special endomorphisms of a finite group. We show how such an endomorphism gives rise to two non-degenerate solutions to the Yang-Baxter equation, solutions which are inverse to each other. We give concrete examples using dihedral, alternating, symmetric, and metacyclic groups. 
	\end{abstract}

\maketitle
\section{Introduction}

In the span of just a few years, Chen-Ning Yang \cite{Yang67}, working in theoretical physics, and Rodney Baxter \cite{Baxter72}, working in statistical mechanics, independently developed an equation which is now commonly referred to as the {\it Yang-Baxter equation}. The Yang-Baxter equation has been found to have applications in many areas of mathematics and physics, including quantum groups \cite{LambeRadford97}, knots \cite{Jimbo89a,Jimbo94}, analysis of integrable systems \cite{Jimbo89b}, and quantum computing \cite{Chen12}.

The Yang-Baxter equation describes a relationship between functions, namely \[(R\otimes 1)(1\otimes R)(R\otimes 1) = (1 \otimes R)(R\otimes 1)(1\otimes R),\] where $R\in \End(V\otimes V)$ for some finite dimensional vector space $V$. Interesting solutions of this equation are difficult to find, so an easier problem was introduced by Drin'feld  in \cite{Drinfeld92}: that of finding {\it set-theoretic solutions}, where $V$ is replaced by a finite set $B$, and tensor product replaced by the cross product. Given such a set-theoretic solution on $B$, one obtains a solution for a vector space with $B$ as a basis. 

Many interesting example of set-theoretic solutions start to appear when $B$ carries some additional structure. For example, if $B=(B,\cdot)$ is an abelian group then Rump \cite{Rump07} showed how to construct solutions to the Yang-Baxter equation by giving $B$ the structure of a {\it left brace}, where $(B,\cdot)$ is endowed with another binary operation $\circ$ such that $(B,\circ)$ is a group and a certain relation holds. The corresponding solutions to the equation were both {\it non-degenerate} (see section \ref{YBEsec}) and {\it involutive} (i.e., self-inverse). 

Later, the concept of left brace was generalized by Guarnieri and Vendramin \cite{GuarnieriVendramin17} to the case where $(B,\cdot)$ is nonabelian, obtaining a {\it (left) skew brace}. Skew braces correspond to non-degenerate solutions of the Yang-Baxter equations which are not non-involutive, though one can easily compute the inverse to any solution \cite{KochTruman20b}.

Thus, identifying set-theoretic solutions to the Yang-Baxter equation reduces to finding skew braces. But skew braces may not be easy to find: they involve putting two  (potentially) different group structures on $B$ that interact with each other in a precise way. A breakthrough was obtained in 2016 by Bachiller \cite{Bachiller16}, who observed a correspondence between skew braces and Hopf-Galois structures on Galois field extensions, a subject which had been studied for 30 years by numerous people--see, e.g., \cite{Kohl98}, \cite{CarnahanChilds99}, \cite{Childs00}, \cite{Byott04b}. The classification of Hopf-Galois structures on a Galois extension with Galois group $ G $ follows from a theorem of Greither and Pareigis, and amounts to finding regular subgroups $N$ of $\Perm(G)$, satisfying a certain stability condition. Given $N\le\Perm(G)$ as above, a skew brace can be constructed where one of the underlying groups is isomorphic to $N$, the other to $G$. 

But Hopf-Galois theory still involves two groups interacting in a very precise way. In 2013, Childs \cite{Childs13} developed a technique using {\it fixed point free abelian  endomorphisms} (definition \ref{fpf} below) of a group $G$ to construct a certain class  of suitable regular subgroups, hence of Hopf-Galois structures.

In this work, we use Childs's theory to construct explicit solutions to the Yang-Baxter equation. If $\psi:G\to G$ is a fixed point free abelian endomorphism, then one can obtain a regular subgroup of $\Perm(G)$. Here, we construct the skew  brace $\B_{\psi}$ associated to $\psi$, and then we find a solution to the Yang-Baxter equation based on $\B_{\psi}$. We then use the theory of opposite braces in \cite{KochTruman20} to develop a second solution, namely, the inverse to the first solution. In short, we prove the following.

\begin{thm}[Theorem \ref{main}]
	Let $G$ be a finite group.
	\begin{enumerate}
		\item Let $\psi:G\to G$ be a fixed point free abelian endomorphism. Let $R_{\psi},R'_{\psi}:G\times G \to G \times G$ be given by
		\begin{align*}
		R_{\psi}(g,h)&=\left(\psi(g^{-1})h\psi(g),\psi(hg^{-1})h^{-1}\psi(g)g\psi(g^{-1})h\psi(gh^{-1}\right))\\
		R_{\psi}'(g,h)&=\left(g\psi(g^{-1})h\psi(g)g^{-1},\psi(h)g\psi(h^{-1})\right)
		\end{align*}
		Then:
		\begin{enumerate}
			\item Both $R_{\psi}$ and $R'_{\psi} $ are non-degenerate, set-theoretic solutions to the Yang-Baxter equation.
			\item The functions $R_{\psi}$ and $R'_{\psi}$ are inverse to each other.
		\end{enumerate} 
		\item Two fixed point free abelian endomorphisms $\psi_1,\psi_2:G\to G$ give {\it equivalent} solutions to the Yang-Baxter equation $R_{\psi_1},\;R_{\psi_2}$ if and only if $\psi_2=\varphi\psi_1\varphi^{-1}$ for some $\varphi\in\Aut(G)$.
	\end{enumerate} 
\end{thm}

In section \ref{bracesec}, we define precisely the concepts of {\it type} of solution and {\it equivalence} of solutions.

The significance of this result lies in its simplicity. Although the proof presented here uses the Greither-Pareigis correspondence for Hopf-Galois structures on separable field extensions and the theory of skew braces, the reader can generate pairs of solutions to the Yang-Baxter equation without a deep understanding of these concepts: one simply needs to be familiar with elementary group theory. 

We start by describing, in greater detail, the Yang-Baxter equation. Then, we give an overview of Greither-Pareigis theory, which we then connect to skew braces which, in turn, give us set-theoretic solutions to the Yang-Baxter equation. Next, we discuss fixed-point free abelian endomorphisms on a finite group $G$ and we explicitly construct a brace corresponding to such an endomorphism. We then prove the main result above (with some minor differences in notation) in Section 5. Finally, we provide examples in which $G$ is dihedral, alternating, symmetric, or metacyclic.

Unless otherwise specified, we assume all groups (and braces) are finite.

\section{The Yang-Baxter Equation and its Solutions}\label{YBEsec}

Let $V$ be a finite dimensional vector space over some field $K$. A {\it solution} to the Yang-Baxter equation is a $K$-linear map $R:V\otimes V\to V\otimes V$ such that 
\[(R\otimes I)(I\otimes R)(R\otimes I)=(I\otimes R)(R\otimes I)(I\otimes R): V\otimes V \otimes V\to V\otimes V \otimes V,\]
where $I:V\to V$ is the identity map. If we view $R\in M_{n^2}(K)$ where $\dim_K V=n$ then $R$ is what is called an {\it $R$-matrix}. As a way to develop non-trivial examples of solutions, in \cite{Drinfeld92}, Drin'feld suggests considering set-theoretic solutions to the YBE, that is, a set $B$ and a function $ R: $ $B\times B\to B\times B$ such that
\[(R\times \mathrm{id})(\mathrm{id}\times R)(R\times \mathrm{id}) = (\mathrm{id}\times R)(R\times \mathrm{id})(\mathrm{id}\times R)\]
holds, where $\mathrm{id}$ is the identity map on $B$. If $B$ is finite, we let $V_B$ be the $K$-vector space with basis $B$; then a set theoretic solution on $B$ gives a vector space solution on $V_B$. 

A set-theoretic solution $R:B\times B\to B\times B$ is said to be {\it non-degenerate} if the functions $f_y,g_x:B\to B$ given by $ R(x,y)=(f_y(x),g_x(y))$ are invertible. As above, we say $R$ is {\it involutive} if $\mathrm{id}$. 

In the work being presented here, the set $B$ will carry the structure of a (finite)  group, say $G$, and the set-theoretic solutions will rely heavily on the group operation. Working with $G$ will allow us to identify $V_G$ naturally with $K[G]$, thereby naturally obtaining $R$-matrices on group algebras.

\section{Regular Subgroups, Braces, and Solutions to the YBE}\label{bracesec}

In this section, we introduce many of the structures needed to prove theorem \ref{main}. We start with the definition of regular $G$-stable subgroups of $\Perm(G)$, which give rise to skew braces, which in turn give rise to solutions to the Yang-Baxter equation.

Let $G$ be a group, and denote by $\Perm(G)$ the group of all permutations of $G$. We say a subgroup $N$ of $\Perm(G)$ is {\it regular} if the action of $N$ on $G$ is free and transitive, that is
\begin{enumerate}
	\item for $\eta\in N, g\in G$ we have $\eta[g]=g$ if and only if $\eta=1_N$; and
	\item for $g,h\in G$ there exists an $\eta\in N$ such that $\eta[g]=h$.
\end{enumerate}

If $N\le\Perm(G)$ is regular then $|N|=|G|$. Additionally, if $N\le \Perm(G)$ with $|N|=|G|$ and either of the above two conditions are satisfied, then $N$ is a regular subgroup. Two fundamental examples of regular subgroups are the images of the left regular representation $\lambda: G \hookrightarrow \Perm(G)$ and the right regular representation $\rho:G\hookrightarrow \Perm(G)$.

In general, it is not necessarily the case that $N\cong G$, however in the examples we will consider here these groups will in fact be isomorphic. 

The subgroup $\lambda(G)\le\Perm(G)$ acts on $\Perm(G)$ by conjugation. We say $N\le\Perm(G)$ is {\it $G$-stable} if this action restricts to an action on $N$. In other words, for all $\eta\in N,g\in G$ we have 
\[^g\eta:=\lambda(g)\eta\lambda(g^{-1})\in N.\]

Certainly, $\lambda(G)$ is $G$-stable, and since $\lambda(G)$ and $\rho(G)$ commute in $\Perm(G)$ we see that $\rho(G)$ is $G$-stable as well.

Subgroups of $\Perm(G)$ which are regular and $G$-stable are essential in identifying Hopf-Galois structures on Galois extensions $L/K$ with $\Gal(L/K)\cong G$. This was first discovered by Greither and Pareigis in \cite{GreitherPareigis87}, but will not be needed here.

We will now introduce the notion of a skew (left) brace, and connect it to the regular, $G$-stable subgroups above. 

A {\it skew left brace} is a triple $\B=(B,\cdot,\circ)$ where $(B,\cdot)$ and $(B,\circ)$ are groups; and for all $a,b,c\in B$ the following condition holds:
\[a\circ(b\cdot c) = (a\circ b)\cdot a^{-1} \cdot (a\circ c),\]
where $a^{-1}$ is the inverse to $a$ in $(B,\cdot)$. We call this condition the {\it brace relation}.

For simplicity, we adopt the following conventions.
\begin{enumerate}
	\item We will refer to $\B$ simply as a {\it brace}. The term ``brace'' was first used by Rump in \cite{Rump07} to describe a triple $(B,\cdot,\circ)$ satisfying the brace relation where $(B,\cdot)$ is abelian. The notion of brace was generalized 10 years later by Guarnieri and Vendramin \cite{GuarnieriVendramin17}.
	\item When no confusion will arise we will suppress the dot and write $ab=a\cdot b$.
	\item While $(B,\cdot)$ and $(B,\circ)$ share the same identity element, their inverses are different in general: we will denote the inverse to $a\in(B,\circ)$ by $\overline{a}$.
\end{enumerate}

The reader should be made aware that there is no universal notation for braces at this point. Not only are the symbols for the two operation not standard, the order in which they appear in the triple is not standard as well. 

Simple examples of braces include $(G,\cdot, \cdot)$, called the {\it trivial brace on $G$}, and $(G,\cdot,\cdot')$ where $(G,\cdot')$ is the opposite group to $(G,\cdot)$, giving the {\it almost trivial brace on $G$}.

As with regular, $G$-stable subgroups of $\Perm(G)$, we do not require $(B,\cdot)\cong (B,\circ)$, however this paper will only deal with braces satisfying this property.

A regular, $G$-stable subgroup $N\le \Perm(G)$ gives rise to a brace as follows. The map $\varkappa:N\to G$ given by $\varkappa(\eta)=\eta(1_G)$ is a bijection. Define  $\B(N)=(N,\cdot,\circ)$, where $\cdot$ is the usual group operation on $N$, and 
\[\eta\circ \pi = \varkappa^{-1}(\varkappa(\eta)\ast_G \varkappa(\pi)),\;\eta,\pi\in N.\]

It is well-known that every brace $(B,\cdot,\circ)$ arises from a regular, $G$-stable subgroup $N\le\Perm(G)$ with $G\cong (B,\circ)$ \cite[Prop. 6.1]{Bachiller16}. In general, the correspondence is not one-to-one \cite[Appendix A]{GuarnieriVendramin17}.

Braces give set-theoretic solutions to the Yang-Baxter equation. For $\B=(B,\cdot,\circ)$ a brace, define $R_{\B}:B\times B\to B\times B$ by
\[R_{\B}=(a^{-1}(a\circ b),\overline{a^{-1}(a\circ b)}\circ a \circ b) \label{YBE}]\tag{1}.\]
Then $R_{\B}$ is a non-degenerate, set-theoretic solution to the YBE. 

\begin{definition}
	Let $\B=(B,\cdot,\circ)$ be a brace. Then $R_{\B}$ is said to be a solution of {\it type $(N,G)$} if $N$ and $G$ are abstract groups with $N\cong(B,\cdot)$ and $G\cong (B,\circ)$. If $N=G$ then we say $R_{\B}$ is a solution of {\it type $G$}. 
\end{definition}

If $(B,\cdot)$ is abelian, then $R_{\B}$ is involutive, i.e., $R_{\B}$ is self-inverse. On the other hand, if $(B,\cdot)$ is non-abelian, two of the authors in \cite[Th. 4.1]{KochTruman20} construct the inverse to $R_{\B}$ through the use of opposite braces. The inverse to $R_{\B}$, denoted $R'_{\B}$, is simply
\[R'_{\B}=((a\circ b)a^{-1},\overline{(a\circ b)a^{-1}}\circ a \circ b). \tag{2}\label{YBEopp}\]
 In this case, $R'_{\B}$ is a solution of the same type as $R_{\B}$. Notice that $R'_{\B}= R_{\B'}$, i.e., it is the solution obtained from \ref{YBE} using the opposite brace to $\B$.

It should be pointed out that braces can give additional set-theoretic solutions. For example, suppose $\phi:\B\to \B$ is a brace isomorphism, i.e., a bijective map which preserves both operations, and $R$ is any solution to the Yang-Baxter equation with underlying set $B$. Then $R_{\phi}:=R(\phi\times\phi):B \times B\to B\times B$ is given by
\begin{align*}
R_{\phi} (a,b) &= (\phi(a)^{-1}(\phi(a)\circ \phi(b)),\overline{\phi(a)^{-1}(\phi(a)\circ \phi(b))}\circ \phi(a) \circ \phi(b))\\
 &= (\phi(a^{-1}(a\circ b)),\phi(\overline{a^{-1}(a\circ b)}\circ a \circ b)),
\end{align*}
so $R(\phi\times\phi)=(\phi\times\phi)R$ and the Yang-Baxter equation is clearly satisfied. However, if one is interested in set-theoretic solutions as a means to find vector space solutions, then $R_{\phi}$ and $R$ can be thought of as being the same solution under a change of basis. Thus if $N_1$ and $N_2$ are regular, $G$-stable subgroups which produce the same brace (up to isomorphism), we think of them as giving {\it equivalent} solutions to the Yang-Baxter equation.

\section{Fixed-Point Free Abelian Endomorphisms and Braces}

The purpose of this section is to explicitly show how to obtain a brace from a fixed point free abelian endomorphism. 

We start with a formal definition of what it means for a an endomorphism to be fixed point free abelian. The term is self-explanatory, however since these endomorphisms are central to this work we wish to be explicit. 

\begin{definition}\label{fpf}
	Let $G$ be a group. An endomorphism $\psi:G\to G$ is said to be {\it fixed point free abelian} if both of the following hold:
	\begin{enumerate}
		\item If $\psi(g)=g$ then $g=1_G$. (Fixed point free)
		\item The image $\psi(G)$ is abelian. (Abelian)
	\end{enumerate}
\end{definition}

Observe that if $\psi$ is an abelian endomorphism, then $\psi(ghg^{-1})=\psi(h)$ for all $g,h\in G$. That is, $\psi$ is constant on conjugacy classes. We will use this fact extensively in section \ref{exSec} when constructing these maps.

For a given group $G$, we will denote by $\FPF(G)$ the set of all fixed point free abelian endomorphisms of $ G $. If $\psi_1,\psi_2\in\FPF(G)$ their composition is typically not fixed point free, so $\FPF(G)$ should be viewed as simply a pointed set, base point $\psi_0$.

\begin{example}\label{triv0}
	Let $G$ be any group. Then the trivial map $\psi_0:G \to G,\; \psi_0(g)=1_G$ is fixed point free abelian.
\end{example}

\begin{example}\label{D40}
		Let $G=D_4=\gen{r,s:r^4=s^2=rsrs=1_G}$ be the dihedral group of order $8$. Define $\psi:G\to G$ by $\psi(r) = \psi(s) = rs$. Since $\psi(G)=\gen{rs}$ this map is clearly abelian. Since the image of $\psi$ contains only one nontrivial point, the only possible fixed point is $rs$; since $\psi(rs)=(rs)^2=1_G$ we get that $\psi$ is fixed point free, hence $\psi\in\FPF(G)$. 
		
		This example is well-known: see \cite[p. 1264, (6)]{Childs13}.
\end{example}

In \cite{Childs13} Childs constructs, from $\psi\in\FPF(G)$, a regular $G$-stable subgroup of $\Perm(G)$, namely $N=\{\lambda(g)\rho(\psi(g)):g\in G\}$. Explicitly,
\[\eta_g[h]=gh\psi(g^{-1}) {\mbox{ for all $ g,h \in G $}}.\]
The reader can check that $N$ is regular, and that $^k\eta_g = \eta_{kgk^{-1}}$ for all $g,k\in G$. 

Different fixed point free endomorphisms can yield the same group $N$. Indeed,  $\psi_1$, $\psi_2\in\FPF(G)$ produce the same regular, $G$-stable subgroup of $\Perm(G)$ if and only if there is a $\zeta\in\FPF(G)$ with $\zeta(G)\le Z(G)$ and such that
\[ \psi_2(g)=\psi_1(g\zeta(g^{-1}))\zeta(g),\;g\in G.\]
In particular, $\psi\in\FPF(G)$ gives the same regular, $G$-stable subgroup of $\Perm(G)$ as $\psi_0$--namely, $\lambda(G)\le\Perm(G)$--if and only if $\psi(G)\le Z(G)$.
 See \cite[Th. 2]{Childs13} for details.

Thus, $N$ gives rise to a brace $\B_{\psi}=(N,\cdot,\circ)$ where the dot operation is the usual operation in $N$, which itself is inherited from $G$: since $\lambda(G)$ and $\rho(G) $ commute in $\Perm(G)$ we have
\[\eta_g\eta_h = \lambda(g)\rho(\psi(g))\lambda(h)\rho(\psi(h)) = \lambda(g)\lambda(h)\rho(\psi(g))\rho(\psi(h))=\lambda(gh)\rho(\psi(gh))=\eta_{gh}.\]
%\begin{align*} 
%\eta_g\eta_h[k] &= \eta(g)\big[hk\psi(h^{-1})\big]\\
%&= g\big(hk\psi(h^{-1})\big)\psi(g^{-1})\\
%&= (gh) k \psi((gh)^{-1})\\
%&=\eta_{gh}[k].
%\end{align*} 
 We will now explicitly compute the circle operation, completing the description of $\B_{\psi}$. Note that the bijection $\varkappa:N\to G$ is $\varkappa[\eta_g] = \eta_{g}[1_G] = g\psi(g^{-1})$. For $g,h\in G$ we have

\begin{align*}
\eta_g\circ \eta_h &= \varkappa^{-1}(\varkappa(\eta_g)\ast_G \varkappa(\eta_h))\\
&= \varkappa^{-1}\big(g\psi(g^{-1})h\psi(h^{-1})\big).
\end{align*}

On the other hand, since $\psi$ is constant on conjugacy classes,
\begin{align*}
\varkappa\big(\eta_{g\psi(g^{-1})h\psi(g)}\big) &= g\psi(g^{-1})h\psi(g)\psi\big((g\psi(g^{-1})h\psi(g))^{-1}\big)\\
&=g\psi(g^{-1})h\psi(g)\psi(h^{-1}g^{-1}) \\
&=g\psi(g^{-1})h\psi(h^{-1}), 
\end{align*}
hence
\[\eta_g\circ \eta_h = \eta_{g\psi(g^{-1})h\psi(g)}. \]

At this point we remark that the map $g\mapsto \eta_g:G\to (N,\cdot)$ is in fact an isomorphism. This allows us to think of the underlying set of our brace as $G$, obtaining:
\begin{proposition}
	Let $\psi:G \to G$ be a fixed point free abelian endomorphism. Then $\B_{\psi}=(G,\cdot,\circ)$ is a brace, where
	\[g\circ h = g\psi(g^{-1})h\psi(g),\;g,h\in G.\]
\end{proposition}

As mentioned earlier, the two groups in a brace need not be isomorphic, however for braces constructed from fixed point free abelian maps, it is necessary that they are.

\begin{corollary}
	Let $\psi\in\FPF(G)$, and let $B_{\psi}=(G,\cdot,\circ)$ be its corresponding brace. Then $(G,\circ)\cong (G,\cdot)$.
\end{corollary}

\begin{proof}
	Define $\phi:(G,\circ)\to (G,\cdot)$ by $\phi(g)=g\psi(g^{-1})$. We claim that $\phi$ is an isomorphism.  Indeed, we have
	\begin{align*}
	\phi(g\circ h)&= \phi( g\psi(g^{-1})h\psi(g))\\
	&= g\psi(g^{-1})h\psi(g)\psi(\psi(g^{-1})h^{-1}\psi(g)g^{-1})\\
	&=g\psi(g^{-1})h\psi(h^{-1})\tag{$\psi$ is abelian}\\
	&=\phi(g)\phi(h)
	\end{align*}
	so $\phi$ is a homomorphism. If $g\in \ker \phi$ then $\phi(g)=g\psi(g^{-1})=1_G$, from which it follows that $g=\psi(g)$. As $\psi$ is fixed point free, this means $g=1_G$.
Thus, $\phi$ is injective, hence an isomorphism.
\end{proof}

Direct calculation establishes the following, which will be needed in the sequel.
\begin{corollary}\label{thecor}
	Let $\psi\in\FPF(G)$. Then for all $g\in G$ we have
	\[\overline{g}=\psi(g)g^{-1}\psi(g^{-1}).\]
\end{corollary}

From \cite[\S 5]{KochTruman20b}, the group of automorphisms $\Aut(G)$ acts on $\FPF(G)$ via conjugation, i.e., $\varphi\psi\varphi^{-1}\in\FPF(G)$ whenever $\varphi\in\Aut(G),\;\psi\in\FPF(G)$. Furthermore, the braces formed by $\psi$ and $\varphi\psi\varphi^{-1}$ are isomorphic: in fact, any regular subgroup of $\Perm(G)$ which gives a brace isomorphic to the one determined by $\psi$ is of this form. 

\begin{definition}
	We say $\psi_1,\;\psi_2\in\FPF(G)$ are {\it brace equivalent} if $\psi_2=\varphi\psi_1\varphi^{-1}$ for some $\varphi\in\Aut(G)$.
\end{definition}

\begin{example}\label{triv}
	Let $G$ be any group, and let $\psi_0:G\to G$ be the trivial map as in example \ref{triv0}. Then 
	\[g \circ h = g \psi_0(g^{-1})h\psi_0(g)=gh\]
	and $\B_{\psi}=(G,\cdot,\cdot)$ is the trivial brace on $G$.
	
	Notice that for any $\varphi\in\Aut(G)$ we have 
	\[\varphi^{-1}\psi_0\varphi(g) = \varphi^{-1}(1_G)=1_G\]
	so $\varphi^{-1}\psi_0\varphi=\psi_0$. Thus, no nontrivial $\psi\in\FPF(G)$ is equivalent to $\psi_0$.
\end{example}

\begin{example}\label{D4}
	Let $G=D_4=\gen{r,s:r^4=s^2=rsrs=1_G}$ be the dihedral group of order $ 8 $?, and let $\psi_1\in\FPF(G)$ be $\psi_1(r) = \psi_1(s) = rs$ as in example \ref{D40}. The circle operation in $\B_{\psi_1}$ is given by
	\begin{align*}
	r^i\circ r^j &= r^i(rs)^ir^j(rs)^i=r^{i+(-1)^ij}\\
	r^i\circ r^js &= r^i(rs)^ir^js(rs)^i=r^{-i+(-1)^ij}s\\
	r^is\circ r^j&= r^i{s}(rs)^{i+1}r^j(rs)^{i+1} = r^{i+(-1)^ij}s \\
	r^is\circ r^js &= r^is(rs)^{i+1}r^js(rs)^{i+1}=r^{2-i+(-1)^ij}.\\
	\end{align*}
\end{example}

\begin{example}\label{other D4}
Let $G=D_4$ as above, and let $\psi_2 : G \to G$ be the endomorphism given by $\psi_2(r)=r^2s$, $\psi_2(s)=1$. Since $\psi_2(G)=\gen{r^2s}$ the map is abelian, and since $\psi(r^2s) = 1_G \ne r^2s$ we get $\psi_2\in\FPF(G)$. The circle operation in ${\B_{\psi_2}}$ is given by
\begin{align*}
r^i \circ r^j &= r^i(r^2s)^ir^j(r^2s)^i = r^{i + (-1)^ij}\\
r^i \circ r^js &= r^i(r^2s)^ir^js(r^2s)^i = r^{i + (-1)^ij}s\\
r^is \circ r^j &= r^is(r^2s)^ir^j(r^2s)^i = r^{i-(-1)^ij}s\\
r^is \circ r^js &= r^is(r^2s)^ir^js(r^2s)^i = r^{i - (-1)^ij}. 
\end{align*}
%Now let $\varphi\in\Aut(G)$ be given by $\varphi(r)=r^3,\;\varphi(s)=rs$. Then
%\begin{align*}
%\varphi^{-1}\psi_1\varphi(r) &= \varphi^{-1}\psi_1(r^3) = \varphi^{-1}(rs) = r^2s =\psi_2(r)\\
%\varphi^{-1}\psi_1\varphi(s) &= \varphi^{-1}\psi_1(rs) = \varphi^{-1}(1_G) = 1_G =\psi_2(s),
%\end{align*}
Now let $\varphi\in\Aut(G)$ be given by $\varphi(r)=r,\;\varphi(s)=r^3s$. Then
\begin{align*}
\varphi^{-1}\psi_1\varphi(r) &= \varphi^{-1}\psi_1(r) = \varphi^{-1}(rs) = r^2s =\psi_2(r)\\
\varphi^{-1}\psi_1\varphi(s) &= \varphi^{-1}\psi_1(r^3s) = \varphi^{-1}(1_G) = 1_G =\psi_2(s),
\end{align*}
so $\psi_2$ is equivalent to the $\psi_1$ found in example \ref{D4}.
\end{example}

\section{Proof Of The Main Result}

We are now well-positioned to prove our main theorem.

\begin{theorem}\label{main}
	Let $G$ be a finite group.
	\begin{enumerate}
		\item Let $\psi\in\FPF(G)$. Let $R_{\psi},R'_{\psi}:G\times G \to G \times G$ be given by
		\begin{align*}
		R_{\psi}(g,h)&=\left(\psi(g^{-1})h\psi(g),\psi(hg^{-1})h^{-1}\psi(g)g\psi(g^{-1})h\psi(gh^{-1}\right))\\
		R_{\psi}'(g,h)&=\left(g\psi(g^{-1})h\psi(g)g^{-1},\psi(h)g\psi(h^{-1})\right).
		%	R_{\B}'(g,h)&=\left( g\psi(g^{-1})h\psi(g)g^{-1},\psi(h\psi(h^{-1}))g\psi(g^{-1})h\psi(g\psi(h)) \right), g,h\in G.
		\end{align*}
		Then:
		\begin{enumerate}
			\item Both $R_{\psi}$ and $R'_{\psi}$ are non-degenerate, set-theoretic solutions to the Yang-Baxter equation of type $G$.
			\item The functions $R_{\psi}$  and $R'_{\psi}$ are inverse to each other.
		\end{enumerate} 
	\item The maps $\psi_1,\psi_2\in\FPF(G)$ give equivalent pairs of solutions to the Yang-Baxter equation $R_{\psi_1},\;R_{\psi_2}$ if and only if $\psi_2=\varphi\psi_1\varphi^{-1}$ for some $\varphi\in\Aut(G)$.
	\end{enumerate} 
\end{theorem}

\begin{proof}
	Any $\psi\in\FPF(G)$ gives rise to a regular, $G$-stable subgroup $N_{\psi}$ of $\Perm(G)$, which in turn corresponds to the brace $\B_{\psi}=(G,\cdot,\circ)$ from before. The solution $R_{\psi}$ can be readily computed from equation (\ref{YBE}):
	\begin{align*}
	R_{\psi}(g,h) &=\big(g^{-1}(g\circ h),\overline{g^{-1}(g\circ h)}\circ g \circ h\big)\\
	&=\big(g^{-1}g\psi(g^{-1})h\psi(g),\overline{g^{-1}g\psi(g^{-1})h\psi(g)}\circ g \circ h\big) \tag{cor. \ref{thecor}} \\
	&=\left(\psi(g^{-1})h\psi(g),(\psi(\psi(g^{-1})h\psi(g)))(\psi(g^{-1})h^{-1}\psi(g))\psi(\psi(g^{-1})h^{-1}\psi(g))\circ g \circ h\right).
	\end{align*}
	The first component now agrees with the statement of the theorem. For the second component,
	\begin{align*}
     &(\psi(\psi(g^{-1})h\psi(g)))(\psi(g^{-1})h^{-1}\psi(g))\psi(\psi(g^{-1})h^{-1}\psi(g))\circ g \circ h\\&=\psi(h)\psi(g^{-1})h^{-1}\psi(g)\psi(h^{-1})\circ g\psi(g^{-1})h\psi(g)\\
     &=(\psi(hg^{-1})h^{-1}\psi(gh^{-1})\psi\big(\psi(gh^{-1})h\psi(hg^{-1})\big) g\psi(g^{-1})h\psi(g) \psi\big(\psi(hg^{-1})h^{-1}\psi(gh^{-1})\big)\\
     &=\psi(hg^{-1})h^{-1}\psi(gh^{-1})\psi(h)g\psi(g^{-1})h\psi(g)\psi(h^{-1})\\
     &=\psi(h\psi(h^{-1}))g\psi(g^{-1})h\psi(g\psi(h))
	\end{align*}
	as desired.

The solution $R_{\psi}'$ can be computed using the theory of opposites as in \cite{KochTruman20}. Given a brace $\B=(B,\cdot,\circ)$ with $(B,\cdot)$ nonabelian one can construct an opposite brace $\B'=(B,\cdot',\circ)$ where $a\cdot' b = b\cdot a$ for all $a,b\in B$. One then obtains a second solution to the YBE, namely
\[(g,h)\mapsto\big((g\circ h)g^{-1},\overline{(g\circ h)g^{-1}}\circ g \circ h\big).\]
One can show that the expression above reduces to $R_{\psi}'$ as in the statement of the theorem; that $R_{\psi}'$ is the inverse to $R_{\psi}$ can be found in \cite[Th. 4.1]{KochTruman20}. 

Finally, to show (2) we use the fact that two regular, $G$-stable subgroups are brace equivalent if and only if their opposites are brace equivalent \cite[Cor. 7.2]{KochTruman20b}.
\end{proof}

\begin{example}\label{trivYBE}
	For any nonabelian group $G$, let $\psi:G\to G$ be the trivial map as in example \ref{triv}. Then we get the trivial brace on $G$, and
	\begin{align*}
	R_{\psi}(g,h) &= \left(\psi(g^{-1})h\psi(g),\psi(hg^{-1})h^{-1}\psi(g)g\psi(g^{-1})h\psi(gh^{-1}\right)) = (h,h^{-1}gh)\\
	R_{\psi}'(g,h) &= \left(g\psi(g^{-1})h\psi(g)g^{-1},\psi(h)g\psi(h^{-1})\right)=(ghg^{-1},g).
	\end{align*}
\end{example}

\begin{example}
	As in example \ref{D4}, let $G=D_4=\gen{r,s:r^4=s^2=rsrs=1_G},\;\psi(r) = \psi(s) = rs$. Using theorem \ref{main} we see, e.g., $R_{\psi}(rs,r^2) = (r^2,rs)$. More generally,
	the values for the solution $R_{\psi}$ are given in the following table:

	\[\begin{array}{c|cccccccc}      & 1        & r          & r^2        & r^3        & s           & rs        & r^2s         & r^3s        \\\hline
1    & (1,1)    & (r, 1)     & (r^2, 1)   & (r^3, 1)   & (s,1)       & (rs,1)    & (r^2s,1)     & (r^3s,1)    \\
r    & (1,r)    & (r^3,r^3)  & (r^2,r)    & (r,r^3)    & (r^2s,r)    & (rs,r^3)  & (s,r)        & (r^3s,r^3)  \\
r^2  & (1,r^2)  & (r,r^2)    & (r^2,r^2)  & (r^3,r^2)  & (s,r^2)     & (rs,r^2)  & (r^2s,r^2)   & (r^3s,r^2)  \\
r^3  & (1,r^3)  & (r^3,r)    & (r^2,r^3)  & (r,r)      & (r^2s,r^3)  & (rs,r)    & (s,r^3)      & (r^3s,r)    \\
s    & (1,s)    & (r^3,s)    & (r^2,s)    & (r,s)      & (r^2s,r^2s) & (rs,r^2s) & (s,r^2s)     & (r^3s,r^2s) \\
rs   & (1,rs)   & (r,r^3s)   & (r^2,rs)   & (r^3,r^3s) & (s,r^3s)    & (rs,rs)   & (r^2s, r^3s) & (r^3s,rs)   \\
r^2s & (1,r^2s) & (r^3,r^2s) & (r^2,r^2s) & (r,r^2s)   & (r^2s,r^3s) & (rs,s)    & (s,s)        & (r^3s,s)    \\
r^3s & (1,r^3s) & (r,rs)     & (r^2,r^3s) & (r^3,rs)   & (s,rs)      & (rs,r^3s) & (r^2s,rs)    & (r^3s,r^3s) \end{array}\]  
A table for $R'_{\psi}$ can be similarly constructed using the formula as in theorem \ref{main}; alternatively one could compute it by explicitly constructing the inverse to the table of values above.
\end{example}

\section{Examples}\label{exSec}

We finish with examples, where we explicitly construct the solutions to the Yang-Baxter equation whose underlying sets are some well-known groups. We have tried to strike a balance between explicitness and simplicity of results.

Sometimes, the use of fixed point free abelian endomorphisms will find {\it all} solutions to the Yang-Baxter equation with underlying group $G$. We explore this topic in the examples below.

\subsection{Abelian Groups}
The theory presented above gives useful results when $G$ is nonabelian. In the case where $G$ is abelian, 
\[g\circ h = g\psi(g^{-1})h\psi(g) = gh\]
giving only the trivial brace. The only solution to the Yang-Baxter equation we construct is
\[R_{\psi}(a,b)=(b,a)\]
as $R'_{\psi}=R_{\psi}$ when $G$ is abelian.

By \cite{Byott96}, our construction will yield all possible solutions to the Yang-Baxter equation if and only if $\gcd(|G|,\phi(|G|))=1$, where $\phi$ is the Euler totient function. Such a group is necessarily cyclic.
	
\subsection{The Dihedral Groups $D_n$}

Generalizing the notation from example \ref{D4}, write $G=D_n=\gen{r,s:r^n=s^2=rsrs=1_G}$. Classification of $\FPF(D_n)$ appears in \cite[\S 5]{Childs13}. Here, we will give the Yang-Baxter solutions given by those maps.

Suppose first that $n$ is odd. Then $\FPF(D_n)$ consists only of the trivial map. Indeed, if $\psi\in\FPF(D_n)$ is nontrivial then $\ker\psi=\gen{r}$. Additionally, $\psi(s)=r^is$ for some $0\le i \le n-1$. But then $\psi(r^is)=r^is$ is a fixed point, which is a contradiction.

Thus, the two solutions to the Yang-Baxter equation we obtain when $n$ is odd are the two from example \ref{trivYBE}. 

By the main theorem in \cite{Kohl20}, the number of regular, $G$-stable subgroups of $D_n$ isomorphic to $n$ for $n$ odd is equal to the number of self-inverse elements in $\mathbb{Z}_n^{\times}$, from which it follows that our process obtains all such subgroups if and only if $p$ is a power of an odd prime.

%We now suppose $n\equiv 2 \pmod 4$, say $n=2m$ for $m$ odd. 
Now suppose that $ n $ is even, and write $ n=2m $. 
Each element of $\FPF(D_n)$ which does not give the trivial regular subgroup $\lambda(G)$ is in at least one of the following four forms.
\begin{enumerate}
	\item $\alpha_i(r)=r^{2i}s,\;\alpha_i(s) = 1,\;0\le i\le m-1$;
	\item $\beta_i(r)=r^{2i+m}s,\;\beta_i(s) = r^m,\;0\le i\le m-1$; 
	\item $\gamma_i(r)=r^{2i+1}s,\;\gamma_i(s) = r^{2i+1}s,\;0\le i\le m-1$;
	\item $\delta_i(r)=r^{2i+1+m}s,\;\delta_i(s) = r^{2i+1}s,\;0\le i\le m-1$.
\end{enumerate}

For $0\le i,j \le n-1,\;\gcd(i,n)=1$ let $\varphi_{i,j}:D_n\to D_n$ be given by $\varphi_{i,j}(r)=r^i,\;\varphi_{i,j}(s)=r^js$. Then $\Aut(D_n)=\{\varphi_{i,j}\}$. If we set $i=1$ we have
\begin{align*}
\varphi_{1,j}^{-1}\alpha_0\varphi_{1,j}(r) &= \varphi_{1,j}^{-1}\alpha_0(r)=\varphi_{1,j}^{-1}(s)=r^{-j}s\\
\varphi_{1,j}^{-1}\alpha_0\varphi_{1,j}(s) &= \varphi_{1,j}^{-1}\alpha_0(r^js)=\varphi_{1,j}^{-1}(s^j),\\
\end{align*}
hence if $j$ is even we get $\varphi_{1,j}^{-1}\alpha_0\varphi_{1,j}=\alpha_{j-2}$, and if $j$ is odd we get $\varphi_{1,j}^{-1}\alpha_0\varphi_{1,j}(s)=r^{-j}s$ and so $\varphi_{1,j}^{-1}\alpha_0\varphi_{1,j}=\gamma_{(-j-1/2)}$. By allowing $j$ to vary we see that $\{\alpha_i,\gamma_j\}$ are all equivalent.

Using $\varphi_{1,j}$ it can also be shown that $\{\beta_i,\delta_j\}$ are all brace equivalent in a very similar manner. Furthermore, in the case $m$ is even the map $\zeta: D_n\to D_n$ given by $\zeta(r)=r^m,\;\zeta(s)=r^m$ is a fixed point free endomorphism whose image is the center of $D_n$. Additionally,
\begin{align*}
\alpha_0(r\zeta(r^{-1}))\zeta(r) &= \alpha_0(r^{1+m})r^m = s^{1+m}r^m=r^ms=\beta_0(r)\\
\alpha_0(s\zeta(s^{-1}))\zeta(s) &= \alpha_0(r^ms)r^m = s^{m}r^m = r^m=\beta_0(s),
\end{align*}
so there is only one nontrivial class. 

%On the other hand, if $m$ is odd we can check that $\ker\alpha_0=\gen{r^2,s}$ has order $2m$, while $\ker\beta_0=\gen{r^2}$ has order $m$, thus $\alpha_0$ and $\beta_0$ are not equivalent by proposition \ref{lpoints}. \ak{I like this argument, but I'm not seeing where I use $m$ is odd.}
%

We have
\begin{align*}
R_{\alpha_0}(r^is^j,r^ks^{\ell})&=(s^ir^ks^{\ell},s^{k+i+\ell}r^{-k}s^ir^is^{i+j}r^ks^{j+k+\ell})\\
R_{\alpha_0}'(r^is^j,r^ks^{\ell})&=(r^is^{i+k}r^ks^{i+\ell+j}r^i,s^kr^is^{j+k})
\end{align*}
and for $m$ odd we have, since $r^m\in Z(D_n)$,
\begin{align*}
R_{\beta_0}(r^is^j,r^ks^{\ell})&=(s^ir^ks^{i+\ell},s^{i+k+\ell}r^{-k}s^ir^is^{i+j}r^ks^{i+k+\ell})\\
R'_{\beta_0}(r^is^j,r^ks^{\ell})&=(r^is^{i+j}r^ks^{i+j+{\ell}}r^{-i},s^kr^is^{j+k}).
\end{align*}

As mentioned above, in \cite{Kohl20}, Kohl enumerates all regular, $G$-stable subgroups of $ \Perm(D_{n}) $ isomorphic to $D_n$. Since a regular, $G$-stable subgroup arising from $\psi\in\FPF(G)$ never gives the same brace as one arising from a regular, $G$-stable subgroup which does {\it not} come from such a $\psi$, this allows us to determine when we have found all the Yang-Baxter solutions of dihedral type. More precisely:

\begin{proposition}
	With the constructions above, we obtain every Yang-Baxter solution of dihedral type if and only if one of the following holds:
	\begin{enumerate}
		\item $n=p^k$ for some prime $p$ and positive integer $k$;
		\item $n=2p^k$ for some odd prime $p$ and positive integer $k$;
	\end{enumerate}
\end{proposition} 
\begin{proof}
	The main result of \cite{Kohl20} indicates that the number $\mathscr{R}(D_n)$ of regular, $G$-stable subgroups of $D_n$ isomorphic to $D_n$ is given by the following formula:
	\[ \mathscr{R}(D_n)= |\Upsilon_n|\cdot \begin{cases} 1 & v_2(n)=0\\ 2m+1 & v_2(n)=1 \\ m+1 & v_2(n)=2 \\ m+2 & v_2(n)>2\end{cases},\]
	where $v_2(n)$ is the $2$-adic valuation of $n$, $n=2m$ where appropriate, and $\Upsilon_n$ is the subgroup of self-inverse elements of $\mathbb{Z}_n^{\times}$. Note that $\Upsilon_n$ is nontrivial for $n>3$, and $|\Upsilon_n|=2$ if and only if $\mathbb{Z}_n^{\times}$ is cyclic, which holds if and only if $n=1,2,p^k$ or $2p^k$ for some odd prime $p$ and $k\ge 1$.
	
	On the other hand, in \cite{Childs13} Childs counts the number of regular, $G$-stable subgroups arising from fixed point free abelian endomorphisms on $D_n$. For our purposes, since the theory of opposites allow us to obtain {\it two} regular, $G$-stable subgroups from each $\psi$, we get that the number of subgroups  $\mathscr{S}(D_n)$ we obtain is
	\[\mathscr{S}(D_n) = \begin{cases} 2 & v_2(n)=0 \\ 4m+2 & v_2(n)=1 \\ 8m+2 & v_2(n)>1 \end{cases}.\]
	
	By comparing these, we see that if $n$ is odd we will get every regular, $G$-stable subgroup (hence, every Yang-Baxter solution) if and only if $|\Upsilon_n|=2$, i.e., $n$ is a power of a prime. If $v_2(n)=1$ we again require $|\Upsilon_n|=2$; since $n\ge 3$ is even this means $n=2p^k$ for some $p$. If $v_2(n)=2$ then we would need $(m+2)|\Upsilon_n|=4m+2$. Since $|\Upsilon_n|$ is necessarily a nontrivial power of $2$ and $4m+2\equiv 2 \pmod 4$ we see that we cannot get all regular, $G$-stable subgroups in this manner. Similarly, if $v_2(n)>2$ then we would need $(m+2)|\Upsilon_n|=8m+2$ which cannot happen.
\end{proof}

\subsection{The Alternating Group $A_4$}

It is well-known that $G=A_4$ can be expressed as a semidirect product $V\rtimes C_3$ where $V$ is the (unique) Sylow $2$-subgroup of $A_4$ and $C_3$ is generated by a $3$-cycle, say $x$. Then 
\[A_4=\{vx^i: 0\le i\le 2, v\in V\}.\]
Let $\psi\in\FPF(G)$ be nontrivial. Then, for some $v\in V$ we have either $\psi(v)=1_G$ or $\psi(v)$ is an element of order $2$, say $\psi(v)=w$. But $v$ and $w$ are conjugate in $A_4$, hence $\psi(w)=\psi(v)=w$, a contradiction. Thus, $\psi(v)=1_G$ for all $v\in V$. 

Since $\psi$ is assumed to be nontrivial, $\psi(x)$ is an element of order $3$ which is not conjugate to $x$. There are two conjugacy classes for the $3$-cycles in $A_4$, and every $3$-cycle is in a different class than its inverse. Thus each conjugacy class has four elements, hence there are four possibilities for $\psi(x)$. One such example is $\psi(x)=x^2$. The reader can verify that all four choices give valid fixed-point free abelian endomorphisms.

We omit the technical details above because each of the nontrivial elements of $\FPF(G)$ give the same brace. To see this, suppose $\psi_1(x)=y$ and $\psi_2(x)=z$. Necessarily, $y$ and $z$ are in the same conjugacy class, so there exists an inner automorphism of $ G $ carrying $\psi_1(x)$ to $\psi_2(x)$. 

%Let $G=\gen{x,y:x^3=y^2=(xy)^3=1_G}$. Then $G\cong A_4$ via the map $x\mapsto (123),\;y\mapsto (12)(34)$. Let $\psi:G\to G$ be a nontrivial fixed point free abelian endomorphism. Then $\psi(y)=1_G$ or $\psi(y)$ is an element of order $2$ in $A_4$, say $\psi(y)=z$; but $y$ and $z$ are conjugate, hence $\psi(z)=\psi(y)=z$, a contradiction. Thus, $\psi(y)=1_G$ and $\psi(x)$ is some element of order $3$, say $\psi(x)=w,\;w\in\{x,x^2,xy, (xy)^2, x^2y, (x^2y)^2, xyx, (xyx)^2 \}$. Since $x$ is conjugate to $x,xy,yx$, and $(xyx)^2$ we cannot have $w=x,xy,yx,(xyx)^2$ without creating a fixed point. Thus there are four possibilities:
%\begin{align*}
%\psi_1(x) &= x^2 & \psi_1(y) &=1\\
%\psi_2(x) &= (xy)^2 & \psi_2(y) &=1\\
%\psi_3(x) &= (yx)^2 & \psi_3(y) &=1\\
%\psi_4(x) &= xyx & \psi_4(y) &=1.\\
%\end{align*}
%It is routine to verify that each of these possibilities satisfy the relations above. Furthermore, 
%for each $1\le i\le 4,\;\psi_i(A_4)\cong C_3$, hence $\psi_1$ is abelian. Finally, $\psi_i$ interchanges the two conjugacy classes for elements of order $3$, hence there are no fixed points, giving our fixed point free abelian maps on $A_4$.
%
%On the other hand, the automorphisms of $A_4$ are all conjugation by some element in $S_4$--indeed, $\Aut(A_4)\cong S_4$. Since all elements of order $3$ are conjugate in $S_4$ we have $\phi_{i,j}\in\Aut(A_4)$ which carries $\psi_i(x)$ to $\psi_j(x)$ (and, certainly, $\psi_i(y)$ to $\psi_j(y)$) for all $1\le i,j\le 4$. Thus these maps are all brace equivalent. 
Including the trivial map, we have
\begin{proposition}
	There are five fixed point free abelian maps on $A_4$ and two brace equivalence classes of fixed point free abelian maps.
\end{proposition}

If we use $\psi(x)=x^2$ the solutions obtained are
\begin{align*}
R_{\psi}(v_1x^i,v_2x^j) &= (x^iv_2x^{j-i},x^iv_2x^{-i}v_1x^{-i}v_2x^{-i-j})\\
R_{\psi}'(x^iv_1,x^jv_2) &= (v_1x^{-i}v_2x^{i-j}v_1,x^{-j}v_1x^{i+j})
\end{align*}
where $v_1,v_2\in V$.

As $\Perm(A_4)$ has $10$ regular, $G$-stable subgroups \cite[Th. 7]{CarnahanChilds99}, we see that the solutions presented above are all of the solutions of type $A_4$.

The situation changes drastically for $A_5$, as we shall see.
\subsection{The Alternating Groups $A_n,n\ge 5$}

It is easy to compute $\FPF(A_n)$ for $n\ge 5$. In fact, we generalize, computing $\FPF(G)$ in the case $G$ is a simple group. This is well known and can be shown by direct computation or by considering Hopf-Galois structures on Galois extensions with simple Galois group \cite{Byott04c}.

\begin{proposition}
	Suppose $G$ is a nonabelian simple group. Then $\FPF(G)$ consists of only the trivial map.
\end{proposition}
\begin{proof}
	If $\psi:G\to G$ is any homomorphism, then $\ker\psi=\{1_G\}$ or $\ker\psi=G$ since $\ker\psi\triangleleft G$. If $\ker\psi=\{1_G\}$ then $\psi(G)=G$ and hence $\psi$ cannot be abelian. Thus $\psi$ is trivial.
\end{proof}

Thus the only solutions to the Yang-Baxter equation obtained via fixed-point-free abelian maps are the solutions found in example \ref{trivYBE}. Unlike in the abelian case, we do get multiple solutions. Trivially, this gives us all the Yang-Baxter equation solutions of type $G$ for $G$ a nonabelian simple group.
\subsection{The Symmetric Groups $S_n,n\ge 5$}

Let $\psi:S_n\to S_n$ be a homomorphism. In a manner similar to the previous section, $\ker \psi=\{1_G\},A_n$, or $S_n$. If we assume $\psi$ is nontrivial and fixed point free, then $\ker \psi=A_n$. Now for $\sigma,\pi\in S_n\setminus A_n$ we have $\sigma\pi^{-1}\in A_n$, hence
\[\psi(\pi) = \psi(\pi\sigma^{-1})\psi(\sigma) = \psi(\sigma),\]
so $\psi$ is constant on odd permutations. Furthermore, $\psi(\sigma^2)=1_G$, hence a nontrivial $\psi$ sends odd permutations to some element of order $2$, say $\tau$. Finally, if $\tau$ is odd then $\psi(\tau)=\tau$ and we have a fixed point. In fact, we have:

\begin{proposition}
	Let $A_n^{(2)}$ be the largest subgroup of $A_n$ of exponent two. Then there is a bijection between $A_n^{(2)}$ and $\FPF(S_n)$ given by $\tau\mapsto\psi_{\tau}$, where
	\[\psi_{\tau}(\sigma)=\begin{cases} 1 & \sigma\in A_n \\ \tau & \sigma\notin A_n \end{cases}.\]  
	Furthermore, $\psi_{\tau_1}$ and $\psi_{\tau_2}$ are brace equivalent if and only if $\tau_1$ and $\tau_2$ have the same cycle structure.
\end{proposition}

\begin{proof}
	From the discussion above it is clear that any $\psi\in\FPF(G)$ must be of this form, and it is routine to show that $\psi_{\tau}\in\FPF(G)$ for any $\tau\in A_n$ with $\tau^2=1_G$. Any $\tau_1,\;\tau_2\in A_n$ with $\tau_1^2=\tau_2^2=1_G$ are conjugate if and only if they have the same cycle structure. Now, if $\tau_2=\xi\tau_1\xi^{-1},\;\xi\in S_n$ then the map $C(\xi):S_n\to S_n$ given by conjugation by $\xi$ is an automorphism. If $\sigma\in S_n$ then  
	\[C(\xi)\psi_1C(\xi)^{-1}(\sigma)= C(\xi)\psi_1(\xi^{-1}\sigma\xi)=C(\xi)(1_G)=1_G=\psi_2(\sigma)\]
 if $\sigma$ is even; and 	\begin{align*}
	C(\xi)\psi_1C(\xi)^{-1}(\sigma) &= C(\xi)\psi_1(\xi^{-1}\sigma\xi) \\
	&= C(\xi)\tau_1\\
	&= \xi\tau_1\xi^{-1}\\
	&=\tau_2\\
	&=\psi_2(\sigma)
	\end{align*}
	if $\sigma\in S_n$ is odd. Thus if $\tau_1,\;\tau_2$ have the same cycle structure then the corresponding maps $\psi_1$ and $\psi_2$ are brace equivalent.
	
	Conversely, for $n\ne 6$ every automorphism of $S_n$ is inner, hence these account for all the brace equivalences. Finally, if $n=6$ every element of order $2$ has the same cycle type, so the result follows.
\end{proof}

\begin{corollary}
	Let $\tau \in A_n$ have order at most two. Then
	\[R_{\tau}(\sigma,\pi)=\begin{cases}
	(\pi,\pi^{-1}\sigma\pi) & \sigma,\pi\in A_n\\
	(\pi,\tau\pi^{-1}\sigma\pi\tau) & \sigma\in A_n,\;\pi\notin A_n\\
	(\tau\pi\tau,\tau\pi^{-1}\tau\sigma\tau\pi\tau) & \sigma\notin A_n,\;\pi\in A_n\\
	(\tau\pi\tau,\pi^{-1}\tau\sigma\tau\pi) & \sigma,\pi\notin A_n\\
	\end{cases},\;
	R'_{\tau}(\sigma,\pi)=\begin{cases}
	(\sigma\pi\sigma^{-1},\sigma)& \sigma,\pi\in A_n\\
	(\sigma\pi\sigma^{-1},\tau\sigma\tau) & \sigma\in A_n,\;\pi\notin A_n\\
	(\sigma\tau\pi\tau\sigma^{-1},\sigma) & \sigma\notin A_n,\;\pi\in A_n\\
	(\sigma\tau\pi\tau\sigma^{-1},\tau\sigma\tau) & \sigma,\pi\notin A_n\\
	\end{cases}
	\]
	are solutions to the Yang-Baxter equation.
\end{corollary}

Here, we get all of the solutions of type $S_n$. This can be seen from \cite[Th. 5]{CarnahanChilds99}, where the authors find two families of regular, $G$-stable subgroups of $\Perm(G)$, each parameterized by elements of order $2$ in $S_n$.

\subsection{The Family of Nonabelian Metacyclic Groups $M_{p,q}$}

For our last set of examples we consider the simplest nonabelian metacyclic groups. Let $p,q$ be primes with $p\equiv 1\pmod q$, and define
\[M_{p,q}=\gen{s,t:s^p=t^q=1_G,\;tst^{-1}=s^d},\] where $d$ is chosen to have order $ q $ modulo $ p $. (Note that the group is independent of the choice of $d$: changing its value merely changes the presentation.) By Sylow theory, the Sylow $p$-subgroup $\gen{s}$ is normal in $M_{p,q}$, however since $sts^{-1} \notin\gen{t}$ %\pt{I get $ sts^{-1} = ts^{d^{-1}-1} $, but it's not in $ \gen{t} $, which I guess is the point}
 the Sylow $q$-subgroups are not normal in $M_{p,q}$. Thus, if $\psi$ is any endomorphism of $M_{p,q}$, then $\ker\psi=\{1_G\},\;\gen{s}$, or $M_{p,q}$. Thus for any $\psi$ fixed point free and abelian, $\ker \psi=\gen{s}$ and hence $\psi(s)=1_G$.

It therefore suffices to determine $\psi(t)$. Since the Sylow $p$-subgroup is normal, an element $s^it^j$ will have order $q$ if and only if $j\ne 0$. Let us write \[\psi(t)=s^i t^j,\;0\le i \le p-1,\;1\le j\le q-1.\]
Then $\psi$ respects the relations in the presentation of $M_{p,q}$, hence is an endomorphism. Since $\psi(M_{p,q})=\gen{s^it^j}\cong C_q$ the map is abelian. Suppose $\psi(s^kt^{\ell})=s^kt^{\ell}$. Then
\[(s^it^j)^{\ell} = s^kt^{\ell}.\]
For this to hold we require $j\ell\equiv \ell\pmod q$, so either $j=1$ or $\ell=0$. If $\ell=0$ then $\psi(s^k)=s^k$, so $k=0$ and we have the trivial element. If $j=1$ then $\psi(s^it)=s^it$ and we have a fixed point. Thus:
\begin{proposition}
	If $\psi\in \FPF(M_{p,q})$ is nontrivial, then $\psi=\psi_{i,j}$, where
	\[\psi_{i,j}(s)=1,\;\psi_{i,j}(t)=s^it^j,\;0\le i \le p-1, 2\le j \le q-1.\]
\end{proposition}
It follows from \cite[Lemma 8.13]{KochTruman20b} that $\psi_{i_1,j_1}$ is brace equivalent to $\psi_{i_2,j_2}$ if and only if $i_1=i_2$. From this we obtain
\begin{corollary}
	The set $\{\psi_j:=\psi_{0,j}:2\le j \le q-1\}$ is a complete set of nontrivial brace classes of fixed point free abelian maps.
\end{corollary}

The endomorphism $\psi_j$ produces the following solutions to the Yang-Baxter equation:
\begin{align*}
R_{\psi_j}(s^kt^{\ell}, s^mt^n) &= (
t^{-j{\ell}}s^mt^{n+j{\ell}}, t^{j(n-{\ell})}(s^mt^n)^{-1}t^{jl}s^kt^{{\ell}-j{\ell}}s^mt^{n+j({\ell}-n)})\\
R'_{\psi_j}(s^kt^{\ell}, s^mt^n) &= (s^kt^{{\ell}-j{\ell}}s^mt^{n+j{\ell}}(s^kt^{\ell})^{-1}, t^{jn}s^kt^{{\ell}-jn}).
\end{align*}

One can see that we get all solutions of metacyclic type by using the regular, $G$-stable classification from \cite{Byott04b}; alternatively, see \cite[\S 8]{KochTruman20b}, which explicitly shows this.

\bibliographystyle{alpha} 
\bibliography{MyRefs}

\end{document}